\documentclass[12pt]{article}

\usepackage{latexsym}
\usepackage{amsfonts}
\usepackage{amscd}
\usepackage{amsmath, amsthm, amssymb}
\usepackage{bbm}
\usepackage[all]{xy}
\usepackage{cite}

\usepackage[dvipsnames,svgnames,table]{xcolor}

\usepackage{bbding}
\usepackage{bbm}
\usepackage[version=3]{mhchem}

\usepackage{caption}
\usepackage[all]{xy}

\usepackage{url}
\usepackage{calligra}

\usepackage{fancyhdr}

\usepackage{footmisc}

\usepackage{algorithmic}

\usepackage{listings}

\usepackage{fancyvrb}

\usepackage{chemfig}
\usepackage[labelfont=bf, singlelinecheck=false,font=small]{caption}
\usepackage{fixltx2e}
\usepackage{amsmath}
\usepackage{eucal}
\usepackage{mathrsfs}
\usepackage{mdwlist}
\usepackage{cleveref}
\usepackage{siunitx}
\usepackage{mathtools}
\usepackage{bm}
\usepackage{amsmath}
\usepackage{float}
\usepackage{placeins}
\usepackage{amssymb, graphics}
\usepackage{makeidx}
\usepackage{multirow}
\usepackage{multicol}
\usepackage[titletoc]{appendix}
\usepackage{graphicx}
\usepackage{epstopdf}
\usepackage{ulem}
\usepackage{ulem}
\usepackage{tikz-cd}
\usepackage{booktabs}
\usepackage{listings}

\title{Cohomology of GKM-sheaves} 

\author{ Ibrahem Al-Jabea, Thomas John Baird}

\newtheorem{thm}{Theorem}[section]
\newtheorem{cor}[thm]{Corollary}
\newtheorem{lem}[thm]{Lemma}
\newtheorem{prop}[thm]{Proposition}
\newtheorem{claim}{Claim}[section]

\theoremstyle{definition}

\newtheorem{defn}{Definition}

\newcommand{\ignore}[1]{}

\newcommand{\lie}[1]{\mathfrak{#1}}

\newcommand{\Z}{\mathbb{Z}}
\newcommand{\R}{\mathbb{R}}
\newcommand{\C}{\mathbb{C}}

\newcommand{\F}{\mathcal{F}}

\newcommand{\E}{\mathcal{E}}
\newcommand{\V}{\mathcal{V}}

\begin{document}


\maketitle

\begin{abstract}
Let $T$ be a compact torus and $X$ be a a finite $T$-CW complex  (e.g. a compact $T$-manifold). In \cite{TB}, the second author introduced a functor which assigns to $X$ a so called GKM-sheaf $\F_X$ whose ring of global sections $H^0(\F_X)$ is isomorphic to the equivariant cohomology $H^*_T(X)$ whenever $X$ is equivariantly formal (meaning that $H^*_T(X)$ is  a free module over $H^*(BT))$. In the current paper we prove more generally that $H^0(\F_X) \cong H^*_T(X)$ if and only if $H_T^*(X)$ is reflexive, and find a geometric interpretation of the higher cohomology $H^n(\F_X)$ for $n \geq 1$. 
\end{abstract}

\section{Introduction}


Let $T=(\mathbb{S}^{1})^{r}$ be a compact torus Lie group and let $X$ be a finite $T$-CW complex (such as a compact, smooth $T$-manifold). GKM theory provides techniques for computing the equivariant cohomology ring  $H^{\ast}_{T}(X) := H^*(ET \times_T X; \C)$. For a large class of $T$-manifolds, now called GKM-manifolds, Goresky, Kottwitz and MacPherson \cite{GKM} showed that the cohomology ring $H^{\ast}_{T}(X)$ can be encoded combinatorially in a finite graph (the GKM-graph or moment graph) with edges labelled by non-trivial weights $\alpha \in \Lambda := Hom(T, \mathbb{S}^{1})$. GKM-theory has since developed in several directions: combinatorially by Guillemin and Zara \cite{GZ1,GZ2, GZ3}, to a broader range of spaces by Guillemin and Holm \cite{GH}, and to equivariant intersection cohomology by Braden and MacPherson \cite{BM} who introduced the notion of $\Gamma$-sheaves on a GKM-graph.

In \cite{TB} the second author introduced GKM-sheaves which provide a unified framework for the above constructions. Given a finite $T$-CW complex $X$, we associate a sheaf $\F_X$ whose ring of global sections $H^0(\mathcal F_X)$ is isomorphic to $ H_T^*(X)$ whenever $H_T^*(X)$ is a free module over the cohomology of a point. In the current paper we improve this result by proving that $H^0(\F_X) \cong H^*_T(X)$ if and only if $H_T^*(X)$ is a reflexive module (equivalently a 2-syzygy). Furthermore we show that $H^n(\F_X) = 0$ for $n \geq 2$ and that 

\begin{thm}\label{bigthm}
	If $X$ is a finite $T$-CW complex and $H_T^*(X)$ is reflexive, then there is a natural exact sequence 
	\begin{equation*}
	0\rightarrow H^{0}(\F_{X})\rightarrow H_T^*(X_0) \xrightarrow{\delta} H_T^{*+1}(X_1,X_0)\rightarrow H^{1}(\F_{X})\rightarrow 0
	\end{equation*} 
where $X_0=X^T$ is the fixed point set, $X_1$ is the union of all orbits of dimension one or less, and $\delta$ is the coboundary map in the long exact sequence of the pair $(X_1,X_0)$.
\end{thm}

\textbf{Outline:} In \S \ref{Equivariant Cohomology}, \S \ref{GKM-Sheaves}, and \S \ref{SCGR} we review background material on equivariant cohomology, GKM sheaves, and sheaf cohomology respectively.  In \S \ref{Cohomology of GKM-Sheaves} we study the cohomology of GKM-sheaves and prove that $H^{n}(\F)=0$ for $n\geq 0$ and produce chain complexes to calculate $H^{1}(\F).$ In \S \ref{GMofC} we study the cohomology of the GKM-sheaf $\F_X$ associated to  $T$-space and interpret it geometrically.

\textbf{Acknowledgements:} This paper is based on the master's thesis of the first author supervised by the second author. We thank thesis examiners Misha Kotchetov and Eduardo Martinez-Pedroza for helpful comments.

\section{Equivariant Cohomology}\label{Equivariant Cohomology}

Let $T \cong (\mathbb{S}^{1})^r$ be a compact torus Lie group. The \textit{universal $T$-bundle}, $ET\rightarrow BT$ is a principal bundle for which $ET$ is contractible and whose base $BT$ is homotopy equivalent to the $r$-fold product space $\C P^{\infty} \times ... \times \C P^{\infty}$. Given a $T$-space $X$, the \textit{Borel homotopy quotient} $X_{hT} := ET \times_T X$ is the total space of the associated fibre bundle 
\begin{equation}\label{ETfib}
 X \xrightarrow{i} ET \times_T X \xrightarrow{\pi} BT.
 \end{equation}
The \textit{equivariant cohomology} of $X$ is the singular cohomology of the homotopy quotient
$$ H_T^*(X) := H^*( X_{hT};\C). $$
We use complex coefficients throughout. More generally, if $Y \subseteq X$ is a $T$-invariant subspace then 
$$ H^{*}_{T}(X,Y)=H^{*}(X_{hT}, Y_{hT}). $$
Given a $T$-space $X$, the constant map to a point $X\rightarrow pt$ is equivariant. The induced morphism $H^*(pt)\rightarrow H_{T}^*(X)$ makes $H_{T}^*(X)$ an algebra over $H^*_T(pt) = H^*(BT)$.  By the Kunneth formula
$$H^*(BT)=H^*(\C P^{\infty} \times ... \times \C P^{\infty}) \cong \mathbb{C}[x_{1},\cdots,x_{r}]$$ where each class $x_{i}$ has degree two. More invariantly, there is a natural isomorphism between $H^*(BT)$ and the ring of complex valued polynomials functions on the Lie algebra $\lie{t}$
$$H^*(BT) \cong \C[\lie{t}]. $$
The \textit{weight lattice} 
\begin{equation}\label{weight}
\Lambda := \{ \alpha: T \rightarrow \mathbb{S}^{1}\}
\end{equation} 
is the set of Lie group homomorphisms from $T$ to $\mathbb{S}^{1}=U(1)$. It forms a group under multiplication and there is a natural injection  
$$ \Lambda \hookrightarrow H^2(BT)$$  
which sends $\alpha \in \Lambda$ to the tangent map $d\alpha : \lie{t} \rightarrow \lie{u}(1) = i\R$ regarded as a homogeneous linear polynomial in $\C[\lie{t}]$. 

The Borel Localization Theorem is central to GKM theory. We require only the following basic version (see \cite{FP}).

\begin{thm}[Localization Theorem]\label{locthm}
Let $X$ be a finite $T$-CW complex with fixed point set $i: X^T \hookrightarrow X$. Then the kernel and cokernel of $i^*: H_T^*(X) \rightarrow H_T^*(X^T)$ are both torsion $H^*(BT)$-modules. In particular if $H_T^*(X)$ is torsion free then $i^*$ is injective.
\end{thm}

  \subsection{Atiyah-Bredon Sequence}
  \label{sec:syzygy}
 Let $R=H^{\ast}(BT) \cong \mathbb{C}[x_{1},\cdots,x_{r}]$. A finitely generated $R$-module $M$ is said to be a $j$-th syzygy if there exists an exact sequence
  $$0\rightarrow M\rightarrow F^{1}\rightarrow F^{2}\rightarrow \cdots F^{j}$$
  where the $\{F^{i}\}_{i \in\{1,..,j\}}$ are finitely generated free $R$-modules.  According to (\cite{AFP} Prop. 2.3):
  \begin{itemize}
  \item $M$ is an $r$-syzygy if and only if $M$ is free,
 \item  $M$ is a 1-syzygy if and only if $M$ is torsion free,
  \item $M$ is a 2-syzygy if and only if $M$ is reflexive, meaning the natural map
  	\begin{equation*}
  	M\rightarrow Hom_{R}(Hom_{R}(M,R),R)
  	\end{equation*}
  	is an isomorphism. 
 \end{itemize}

Given a $T$-space $X$, define $X_{i}$ to be the union of all orbits of dimension less than or equal to $i$, $$X_{i}:=\{x\in X \,|\, dim(T \cdot x)\leq i\}.$$ We call $X_{i}$ the $i$-\textbf{skeleton} of $X$. In particular, $X_{-1}=\phi$, $X_{0}=X^{T}$, and $X_{r}=X$.
 
The following is due to Allay-Franz-Puppe (\cite{AFP} Theorem 5.7).
  \begin{thm}\label{sys thrm}
  	 Let $j\geqslant 0$ and let $T$ be a torus of rank $r$, and $X$ be a finite $T$-CW complex. Consider the sequence
  	\begin{equation}\label{at sequence}
  	0\rightarrow H_{T}^*(X)\rightarrow H_{T}^*(X_{0})\xrightarrow{\delta} H_{T}^{*+1}(X_{1},X_{0})\xrightarrow{\delta_{1}}\cdots \xrightarrow{\delta_{r}} H_{T}^{*+r}(X_{r},X_{r-1})\rightarrow 0,
  	\end{equation}
  	where ${\delta_{i}}$ is the boundary map of the triple $(X_{i-1}, X_{i}, X_{i+1})$.
	Then (\ref{at sequence}) is exact for all the positions $i\leqslant j-2$ if and only if ${H^{*}_{T}}(X)$ is j-th syzygy.  
In particular, the sequence 
\begin{equation}
\label{chang sequence}
0\rightarrow H_T^*(X)\rightarrow H_T^*(X_0) \xrightarrow{\delta} H_T^{*+1}(X_1,X_0)\end{equation} is exact if and only if $H_T^*(X)$ is a 2-syzygy if and only if $H_T^*(X)$ is reflexive.
 \end{thm}
 The sequences (\ref{at sequence}) and (\ref{chang sequence}) are known as the Atiyah-Bredon sequence and the Chang-Skjelbred sequence respectively. Observe that if (\ref{chang sequence}) is exact, then 
	 \begin{equation*}
	 H_{T}^*(X)\cong ker(\delta).
	 \end{equation*}
GKM theory is concerned with calculating $ker(\delta).$

\section{GKM-Sheaves}\label{GKM-Sheaves}
Recall the weight lattice $\Lambda := Hom( T, \mathbb{S}^{1})$ from (\ref{weight}).  Declare two weights $\alpha, \beta$ to be projectively equivalent if $\alpha^n = \beta^m$ for some $m,n \in \Z$. The set of \textit{projective weights} $\mathbb{P}(\Lambda)$ is the set of non-zero weights in $\Lambda$ modulo projective equivalence. The elements of $\mathbb{P}(\Lambda)$ are in one to one correspondence with the codimension one subtori of $T$ by the rule 
$$\alpha \in \mathbb{P}(\Lambda)\leftrightarrow \ker_{0}(\tilde{\alpha})\leq T$$
where $\tilde{\alpha}\in\Lambda$ is a representative of $\alpha$, and $\ker_{0}(\tilde{\alpha})$ is the identity component of the kernel of  $\tilde{\alpha}:T\longrightarrow S^{1}. $ We denote $\ker(\alpha)=\ker_{0}(\tilde{\alpha})$.

\begin{defn}
	A \textbf{GKM-hypergraph} $\Gamma$ consists of:
	\begin{enumerate}
		
		\item A finite set of vertices $\V={\cal V}_{\Gamma}$.
		\item An equivalence relation $\sim_{\alpha}$ on ${\cal V}$ for each $\alpha\in\mathbb{P}(\Lambda)$.
	\end{enumerate}
	
\end{defn}

Given a GKM-hypergraph $\Gamma$, the set of hyperedges is defined to be
$${\cal E}={\cal E}_{\Gamma}  := \{ (S, \alpha) \in \wp(\V) \times \mathbb{P}(\Lambda)\; |\;S \;is\; an\; equivalence\; class \;for\; \sim_{\alpha}\}$$
where $\wp(V)$ is the power set of $\V.$ We have projection maps
\begin{itemize} 
	\item $a: {\cal E} \rightarrow \mathbb{P}(\Lambda)$~~~~\text{the \textit{axial function}, and} 
	\item $I: {\cal E} \rightarrow \wp(\cal V)$ ~~~\text{the \textit{incidence function}.}
\end{itemize}
In particular, each hyperedge $e\in {\cal E}$ has associated to it a projective weight $a(e)$ and a non-empty subset $I(e)\subseteq \V$. We say a vertex $v\in \V$ is \textit{incident} to $e\in {\cal E}$ if $v\in I(e)$. Given $\alpha \in \mathbb{P}(\Lambda)$ denote by $\E_{\alpha} := \{ e \in \E~|~a(e) = \alpha\}$.

Let Top$(\Gamma)$ to be the topological space with underlying set $\V \cup {\cal E}$ generated by basic open sets $U_{v}=\{v\}$ for $v\in {\cal{V}}$, and $U_{e}=\{e\}\cup I(e)$ for $e\in {\cal{E}}.$ Observe that for each $x \in Top(\Gamma)$, the set $U_x$ is smallest open set containing $x$.

\begin{defn}\label{gkm sheaf} 
Let $R := H^*(BT) \cong \C[\lie{t}]$. A GKM-sheaf $\cal F$ is a sheaf of finitely generated, $\mathbb{Z}-$graded $R$-modules over Top$(\Gamma)$, satisfying the following conditions.
	\begin{enumerate} 
		\item  $\cal F$ is locally free (i.e, for every basic open set $U_{x}$, the stalk $\F(U_{x})=\F_{x}$ is a free $R$-module).
		\item For every hyperedge $e \in {\cal E}_{\Gamma},$ the restriction map $res_e: {\cal F}(U_e) \rightarrow {\cal F}(I(e))$ is an isomorphism upon inverting $a(e)$:$${\cal F}(U_e) \otimes_R R[a(e)^{-1}] \cong {\cal F}(I(e)) \otimes_R R[a(e)^{-1}].$$
		\item The restriction map $res_e: {\cal F}(U_e) \rightarrow {\cal F}(I(e))$ is an isomorphism for all but a finite number of $e \in {\cal E}_{\Gamma}.$
	\end{enumerate}
	\end{defn}
The main motivating example is the GKM-hypergraph $\Gamma_{X}$ and GKM-sheaf $\F_X$ associated to a finite $T$-CW complex $X$. The vertex set of $\Gamma_X$ is the set $ \V_X := \pi_0(X^T)$ of path components of the $T$-fixed point set $X^T$. Define $v_{1}\sim_{\alpha} v_{2}$ if and only $v_1$ and $v_2$ lie in the same path component of the fixed point set $X^{\ker(\alpha)}$. The hyperedges $ e \in \E_{\alpha} := \{ e \in \E ~|~ a(e) = \alpha\}$ therefore correspond to path components of $X^{\ker(\alpha)}$ that intersect $X^T$ non-trivially. 

Define the GKM-sheaf $\F_{X}$ over $\Gamma_{X}$, as follows. The stalk at a vertex $v \in \pi_0(X^T)$ is  $$\F_{X}(U_{v}):=H_{T}^{\ast}(v),$$ and at a hyperedge $e \in \pi_0(X^{\ker{(\alpha)}})$ is 
$$\F_{X}(U_{e})=\F_{X}( e\cup I(e))=H_{T}^{\ast}(e)/t,$$ 
where $t$ is the torsion submodule of $H_{T}^{\ast}(e).$ The sheaf restriction maps $res_{e}:\F_{X}(U_{e})\rightarrow \F_{X}(I(e))$ are identified with the natural map $H_{T}^{\ast}(e)/t \rightarrow H_{T}^{\ast}(e^{T})$ which is well defined because  $e^{T}\subset e$ and $H_{T}^{\ast}(e^{T})$ is torsion free. This data completely determines $\F_X$.

The following result (Proposition 2.7 in \cite{TB}), relates the degree zero sheaf cohomology of $\F_{X}$ with the equivariant cohomology of $X$.

\begin{prop}\label{fitting}
	Let $X$ be a finite  $T$-CW complex. The space of global sections $H^{0}
	(\F_{X} ) $ fits into an exact
	sequence of graded $R$-modules
	\begin{equation}0\rightarrow H^{0}(\F_{X})\xrightarrow{r} H_T^*(X_0) \xrightarrow{\delta} H_T^{*+1}(X_1,X_0).\end{equation}
\end{prop}

We obtain a generalization of the main result of \cite{TB}, which was originally proven only when $H_T^*(X)$ is a free module.

\begin{cor}
	Let $X$ be a finite $T$-CW complex. If $H_T^*(X)$ is reflexive, then
	 $$H_T^*(X)\cong H^0(\mathcal F_X).$$ 
\end{cor}
\begin{proof}
	Combine Proposition \ref{fitting} with the Chang-Skjelbred sequence (\ref{chang sequence}) which holds if $H_T^*(X)$ is reflexive.
	\end{proof}
For later use, we state the following lemmas from \cite{TB}.
\begin{lem}\label{lemma local}
	If $X$ is a finite $T$-CW complex and $H \subset T$ is a codimension one subtorus, then $H^{\ast}_{T}(X^{H}) $ is
	the direct sum of a free and a torsion $R$-module. If $H^{\ast}_{T}(X) $ is torsion free, then $H^{\ast}_{T}(X^{H}) $ is free.
\end{lem}
\begin{proof}
This is Lemma 2.6 in \cite{TB}.
\end{proof}
\begin{lem}
	\label{decompose lemma}
	Let $X^{\prime}_{1}$ be the union of those components of $X_{1}$ which do not intersect with $X_{0}$. Then $H_T^{*}(X_1,X_0)$ decomposes into 
	
	\begin{equation}\label{decomp}
	H_T^{*}(X_1,X_0)\cong \bigoplus_{e\in \E}H_T^{*}(e,e^{T})\oplus H^{*}_{T}(X^{\prime}_{1}).
	\end{equation}
\end{lem}
\begin{proof}
This is Proposition 2.7 in \cite{TB}.
\end{proof}

\section{Sheaf Cohomology using the Godement Resolution}\label{SCGR}

We summarize material from Iversen \cite{I}. Given a sheaf ${\cal F}$ over a topological space $Y$, define the sheaf $C^{0}{\cal F}$ which sends open sets $U\subseteq Y$ to be the product of stalks
\begin{equation}\label{god}
C^{0}{{\cal F}(U)}=\prod_{y \in U}{{\cal F}_y}
\end{equation}
and whose restriction morphisms are given by projection. There is a natural monomorphism of sheaves,
\begin{equation}\label{cokdefF}
 \F\longrightarrow C^{0}{\cal F}
 \end{equation}
which sends $s \in \F(U)$ to the product of germs $(s_y)_{y \in U} \in C^0(\F)(U)$. Construct sheaves $\F^{n}$ for all $n\geqslant 0$ inductively by setting $\F^0 := \F$ and setting $\F^{n}$ equal to the cokernel sheaf of the natural monomorphism $\F^{n-1} \rightarrow C^0\F^{n-1}$ for all $n \geq 1$. Denote $C^{n}{\cal F}= C^{0}\F^{n}$ . By construction we get short exact sequences of sheaves
\begin{equation*}
0\rightarrow {\cal F}^{n}\xrightarrow{f_n} C^{n} {\cal F}\xrightarrow{g_n}{\cal F}^{n+1}\rightarrow 0
\end{equation*}
for all $n\geq 0$.
Let $d_n:= f_{n+1}\circ g_n$ be the composition,
\begin{equation}\label{bodrymap}
C^{n}{\cal F}\xrightarrow{g_n}{\F^{n+1}}\xrightarrow{f_{n+1}} C^{n+1}{\cal F}
\end{equation}
\begin{thm}
	The sequence of sheaves
\begin{equation}\label{cod res}
0\rightarrow{\cal F}\rightarrow C^{0}{\cal F}\xrightarrow{d_0} C^{1}{\cal F}\xrightarrow{d_1}\dots \xrightarrow{d_{n-1}} C^{n}{\cal F}\xrightarrow{d_n}\cdots
\end{equation}
is exact. It is called  the \bf{Godement resolution} of $\F$.
\end{thm}
\begin{proof}
See \cite{I}.
\end{proof}

Given an open set $U\subseteq Y$, define the chain complex  
\begin{equation*}
 0\xrightarrow{d_{-1}} C^{0}{\cal F}(U)\xrightarrow{d_{0}} C^{1}{\cal F}(U) \xrightarrow{d_{1}} \dots \xrightarrow{d_{n}} C^{n}{\cal F}(U) \rightarrow \cdots.
\end{equation*}
which satisfies $d_{n}\circ d_{n-1}=0$ for all $n\geq 0$. Define the degree $n$ cohomology of $\F$ on $U$ by $$H^{n}(U;\F):=\dfrac{\ker(d_{n})}{\text{im}(d_{n-1})}.$$
We use shorthand $H^n(\F) := H^n(Y,\F)$. The sequence 
$$0\rightarrow \F(U) \rightarrow C^{0}{\cal F}(U) \xrightarrow{d_{0}} C^{1}\F(U)$$
is exact, which implies $H^{0}(U,\F)=\F(U).$

\begin{defn}\label{aloeublaorcebuao} For a sheaf $\cal{F}$ on $Y$ and a closed subset $A\subseteq Y$ we define 
\begin{equation*}
	\Gamma_{A}(\F)=\{ s\in \F(Y)  | supp(s)\subseteq A\} 
\end{equation*}
	where $supp(s)=\{y \in Y \:|\:s_{y}\neq0\}.$ If $A = \emptyset$ we write $\Gamma_{\emptyset}(\F) = \Gamma(\F) = \F(Y)$.
\end{defn}

Given the chain complex 
\begin{equation*} 
0\xrightarrow{d_{-1}}\Gamma_{A}({C^{0}{\cal F}})\xrightarrow{d_{0}} \Gamma_{A}({C^{1}{\cal F}})\xrightarrow{d_{1}}\dots. \xrightarrow{d_{n}} \Gamma_{A}(C^{n}{\cal F})\rightarrow\cdots.
\end{equation*}
define the local cohomology $H_{A}^{n}(Y,\F):=\dfrac{\ker(d_{n})}{\text{im}(d_{n-1})}.$ 
\begin{prop}\label{local co}
	Let $A$ be closed in $Y$. A sheaf ${\cal F}$ on $Y$ gives rise to a long exact sequence
	$$ 0\rightarrow H^{0}_{A}(Y,{\cal F})\rightarrow H^{0}(Y,{\cal F})\rightarrow H^{0}(Y-A,{\cal F})\xrightarrow {\delta}H^{1}_{A}(Y,{\cal F})\rightarrow H^{1}(Y,{\cal F}) \rightarrow\cdots
	$$
\end{prop}

\begin{proof}		
	See Proposition 9.2 in \cite{I}.
\end{proof}

\section{Cohomology of GKM-Sheaves}\label{Cohomology of GKM-Sheaves} 

\subsection{The Godement Chain Complex for GKM-Sheaves}

\begin{prop}\label{higher co}
	If $\Gamma$ is a GKM-hypergraph and ${\cal F}$ is a sheaf on $Top(\Gamma)$, then $C^{n}\F=0$,  for all $n\geq 2$.
\end {prop}
	\begin{proof}
		
		The basic open sets of $Top(\Gamma)$ are
		\begin{itemize}
			\item[(i)] $U_v:=\{v \} $ for vertices $ v \in \V$.
			\item[(ii)] $U_{e}:=\{e \}\cup  I(e)$ for hyperedges $ e\in \E$.
		\end{itemize}
Given a sheaf ${\cal F}$ on $Top(\Gamma)$, for each vertex $v$ we have 
		\begin{equation*}
		(C^{0}{{\cal F}})_{v}:=\prod_{x\in U_{v}}{{\cal F}_{x}}={\cal F}_{v}. 
		\end{equation*}  
The induced morphism ${\cal F}_v \rightarrow (C^{0}{{\cal F}})_{v} ={\cal F}_v$, is an isomorphism, so the cokernel ${\cal F}_{v}^{1}= C^1\F_v$ is zero. Similarly, for all $n\geq 1$,
		\begin{equation}\label{vstalkszero}
		\F_{v}^{n}= C^n\F_v =0.
		\end{equation}
		
For each hyperedge $e$, $\F^{2}_{e}$ is the cokernel of the product of restriction maps
\begin{equation*}\label{f2}
		\F^{1}_{e}\rightarrow  \prod_{x\in U_{e}}{{\cal F}_{x}} \cong \F^{1}_{e}\times \F^{1}_{v_{1}}\times...\times \F^{1}_{v_{k}}.
\end{equation*} 
This is an isomorphism because $\F^{1}_{v_i}=0$, so the cokernel $\F^2_e \cong 0$ . We conclude $\F^{2}=0$ since all of its stalks vanish and consequently $C^{n}\F=0$ for all $n\geq2$. 
\end{proof}
	
	\begin{cor}
		If $\cal F$ is a sheaf over $Top(\Gamma)$, then $H^n({\cal F}) = 0$ for $n \geq 2$ and we have a natural exact sequence  $$0 \rightarrow H^0(\F) \rightarrow \Gamma (C^{0}{{\cal F}})\stackrel{\delta}\longrightarrow\Gamma (C^{1}{{\cal F}}) \rightarrow H^1(\F) \rightarrow 0.$$  
	
	\end{cor}
	\begin{proof}
		By Proposition \ref{higher co}, $ C^{n}({\cal F})={\cal F}^{n}=0$ for all $ n\geq 2$, the chain complex for $H^n({\cal F})$ is		
		\begin{equation}\label{chain complex}
	0 \rightarrow \Gamma (C^{0}{{\cal F}})\stackrel{\delta}\longrightarrow\Gamma (C^{1}{{\cal F}})\rightarrow 0 \rightarrow 0 \rightarrow \cdots.
		\end{equation}
	\end{proof}
	
We want a more concrete description of (\ref{chain complex}).

\begin{lem}\label{stalk edge}
	The stalks of $\F^1$ are as follows: 
	$\F_{v}^1=0$ for all vertices $v$, and 
	 $$\F^1_{e} \cong   \prod_{i=1}^{k}\F_{v_{i}}=\F(I(e))$$ 
	 for all hyperedges $e$, where $I(e) = \{v_{1},v_{2},....v_{k} \}$ is the set of vertices incident to $e$.
\end{lem}
\begin{proof}
That $\F^1_{v}=0$ was proven in (\ref{vstalkszero}). Given a hyperedge $e$, recall $U_{e}:=\{e \}\cup  I(e)$ .  The stalk $\F^{1}_e$ is the cokernel of the map 
	\begin{equation*}
	\F_{e}\xrightarrow{\varepsilon} \prod_{x\in U_e}{{\cal F}_{x}}=\F_{e}\times\F_{v_{1}}\times\cdots\times\F_{v_{k}}
	\end{equation*} 
where ${\varepsilon}(s_{e})=(s_{e},res_{(e,v_{1})}(s_{e}),...,res_{(e,v_{k})}(s_{e}))$, where $res_{(e,v)}: \F_e \rightarrow \F_v$ is the sheaf restriction map associated to $U_v \subseteq U_e$. Because the first coordinate function of $\varepsilon$ is the identity map on $\F_{e}$ we obtain the isomorphism
		$$\F_{e}^{1} := coker(\varepsilon) \cong \prod_{i=1}^{k}\F_{v_{i}}$$
simply by projecting onto the remaining factors. 	
\end{proof}

	\begin{prop}\label{last prop}
		Let $\F$ be a sheaf over $Top(\Gamma)$. Then there is a commutative diagram,

	 \[
	\xymatrix{
		\Gamma(C^{0}\F) \ar@{->}[rr]^{\displaystyle \delta} \ar@{->}[dd]^{\cong}
		&& \Gamma(C^{1}\F) \ar@{->}[dd]^{\cong}    \\ \\
		\prod \limits_{{x\in \V \cup \E}}\F_{x}
		  \ar@{->}[rr]_{\displaystyle \widetilde{\delta}} && \prod\limits_{{e\in \E }}\prod\limits_{{v\in I(e) }} \F_{v}
	}
	\]
		where the vertical maps isomorphisms, $\delta$ is as above and  $\tilde{\delta}$  sends  $s=(s_{x})_{x\in\V\cup\E}$ to $\tilde{\delta}(s)$ with factors
		\begin{equation}
		\tilde{\delta}(s)_{(e,v)}= res_{(e,v)}(s_e)- s_{v}.
		\end{equation}
		In particular, $H^{0}(\F)\cong\ker(\widetilde{\delta})$ and $H^{1}(\F)\cong coker(\widetilde{\delta})$.
	\end{prop}
	\begin{proof}
The isomorphism $ \Gamma(C^{0}\F)  \cong \prod \limits_{{x\in \V \cup \E}}\F_{x} $ is the defining identity (\ref{god}). 
By Lemma \ref{stalk edge} we have $$\Gamma(C^{1}\F):=\prod\limits_{x\in \V\cup \E}\F^1_{x}=\prod\limits_{e\in \E}\F^1_{e}\cong \prod_{e\in \E}\prod_{i=1}^{k}\F_{v_{i}}.$$
The formula for $\tilde{\delta}$ is obtained by chasing through definition \ref{bodrymap}.
\end{proof}

\begin{prop}\label{change}
Let $\F$ be a GKM-sheaf and let $ \E^{nd}  \subseteq \E$ be the finite set of hyperedges for which $res_e$ is not an isomorphism. Then there is an exact sequence  
\begin{equation}\label{freesequence}
0 \rightarrow H^{0}(\F) \rightarrow \bigoplus\limits_{x\in {\E^{nd}}}{{\cal F}_{x}}\xrightarrow{\beta} \bigoplus\limits_{e\in {\E^{nd}}}\bigoplus\limits_{v\in I(e)}\F_{v} \rightarrow H^{1}(\F) \rightarrow 0
\end{equation}
where $\beta$ is the morphism of finitely generated free $R$-modules defined by $$\beta(s)_{(e,v)}=res_{(e,v)}(s_{e})-s_{v}.$$ 
\end{prop}

\begin{proof}

	Let $\mathcal{E}^d := \mathcal{E} \setminus \mathcal{E}^{nd}$. We have a commuting diagram of $R$-modules with exact rows
	\[
	\xymatrix@C+1em@R+1em{ 
		0 \ar[r] &\prod\limits_{x\in {\E^{d}}}{{\cal F}_{x}} \ar[r]^{\psi} \ar^{\gamma}[d] & \prod\limits_{x\in {\V\cup \E}}{{\cal F}_{x}} \ar[r]^{\phi} \ar^{\widetilde{\delta}}[d] & \bigoplus\limits_{x\in {\E^{nd}}}{{\cal F}_{x}} \ar[r] \ar^{\beta}[d]  & 0\\
		0 \ar[r] & \prod\limits_{e\in {\E^{d}}}\prod\limits_{v\in I(e)}\F_{v} \ar[r]^{\psi^{\prime}} & \prod\limits_{e\in {\E}}\prod\limits_{v\in I(e)}\F_{v} \ar[r]^{\phi^{\prime}} & \bigoplus\limits_{e\in {\E^{nd}}}\bigoplus\limits_{v\in I(e)}\F_{v}\ar[r] &  0 .
	}
	\]
	Where $\phi$, $\phi^{\prime}$ are projections, $\psi$, $\psi^{\prime}$ are inclusions, and $\gamma$ is defined by commutativity.  
	 By the Snake Lemma there is an exact sequence 
	\begin{equation*}
	0\rightarrow\ker{\gamma}\xrightarrow{{\overline{\psi}}} \ker(\widetilde{\delta})\xrightarrow{{\overline{\phi}}}\ker(\beta)\rightarrow coker(\gamma)\xrightarrow{{\overline{\psi^{\prime}}}} coker(\widetilde{\delta})\xrightarrow{{\overline{\phi^{\prime}}}}coker(\beta)\rightarrow 0.
	\end{equation*}
	It is clear by definition of $\E^{d}$ that $\gamma$ is an isomorphism. Thus, 
	\begin{equation*}
	0\rightarrow0\rightarrow \ker(\widetilde{\delta})\xrightarrow{{\overline{\phi}}}\ker(\beta)\rightarrow 0\rightarrow coker(\widetilde{\delta})\xrightarrow{{\overline{\phi^{\prime}}}}coker(\beta)\rightarrow 0
	\end{equation*}
	is exact so  ${\overline{\phi}}$ and ${\overline{\phi^{\prime}}}$ are isomorphisms. Compare Proposition \ref{last prop}.
\end{proof}

\begin{cor}\label{reflexive prop}
	If $\F$ is a GKM-sheaf, then $H^{0}(\F)$ is reflexive.
\end{cor}
\begin{proof}
We see from (\ref{freesequence}) we have an exact sequence $0 \rightarrow H^0(\F) \rightarrow F_0 \rightarrow F_1$ where $F_0$ and $F_1$ are finitely generated free modules so $H^{0}(\F)$ is a 2-syzygy, hence reflexive.
\end{proof}

	\subsection{Local Cohomology of a GKM-Sheaf}\label{tobeusedrain}
	 Let $\F$ be an GKM-sheaf over $Top(\Gamma)=\V\cup \E$. The set of vertices $\V$ is an open set and the set of edges $\E$ is a closed set, so we obtain a long exact sequence by Proposition \ref{local co}
	$$0 \rightarrow H^{0}(Top(\Gamma),\F)\rightarrow H^{0}(\V,\F)\rightarrow H^{1}_{\E}(Top(\Gamma),\F) \rightarrow H^{1}(Top(\Gamma),\F)\rightarrow \cdots.  $$
Since $\V$ is discrete, $H^{i}(\V,\F)=0$ for all $i\geq 1$ and
	\begin{equation}\label{local equation}
	0 \rightarrow H^{0}(Top(\Gamma),\F)\rightarrow H^{0}(\V,\F)\rightarrow H^{1}_{\E}(Top(\Gamma),\F) \rightarrow H^{1}(Top(\Gamma),\F)\rightarrow 0
	\end{equation}
is exact. 

\begin{lem}\label{wasclaim}
	The Godement chain complex for $H^{*}_{\E}(Top(\Gamma),\F) $ is given by 
	\begin{equation}
	0\rightarrow \prod\limits_{e\in \E}\F_{e}\xrightarrow{\prod\limits_{e\in \E} res_{e}} \prod\limits_{e\in \E}\prod\limits_{v\in I(e)}\F_{v}\rightarrow 0. 
	\end{equation}
\end{lem}
\begin{proof}
Applying Definition \ref{aloeublaorcebuao} and Lemma \ref{stalk edge} we have
\begin{eqnarray*}
\Gamma_{\E}(C^{0}\F)& = &\{s\in \Gamma (C^{0}\F)\:|\:\:s_{v}=0, \forall v\in \V \}=\prod\limits_{e\in \E}\F_{e}\\
\Gamma_{\E}(C^{1}\F) &= & \{s\in \Gamma (C^{1}\F)\:|\:\:s_{v}=0, \forall v\in \V \}=\prod\limits_{e\in \E}\prod\limits_{v\in I(e)}\F_{v}.
\end{eqnarray*}	
and the boundary map $ \Gamma_{\E}(C^{0}\F)\rightarrow \Gamma_{\E}(C^{1}\F)$ is the natural one. 
\end{proof}

\begin{prop}\label{local e}
	 If $\Gamma$ is a GKM-hypergraph and $\F$ is a GKM-sheaf on $Top(\Gamma)$, then $$H^{0}(\V,\F)\cong \bigoplus\limits_{v\in \V}\F_{v} \text{, and}$$ $$H^{1}_{\E}(Top(\Gamma),\F)\cong \bigoplus\limits_{e\in \E^{nd}}\text{coker}( res_e). $$ 
\end{prop}
\begin{proof}
	Since $\V$ is discrete, $H^{0}(\V)=\prod\limits_{v\in \V}\F(\{v\})=\bigoplus\limits_{v\in \V}\F_{v}$.  
By Lemma \ref{wasclaim} we get
$$	H^{1}_{\E}(Top(\Gamma),\F) = \text{ coker}(\prod\limits_{e\in \E} res_{e})   =\prod_{e\in \E^{nd}} \text{ coker}( res_e).$$

By Definition \ref{gkm sheaf},  $\text{ coker}( res_e)\neq 0$ only for  $e$ in a finite subset $ \mathcal{E}^{nd} \subseteq \mathcal{E}$ so $$\prod_{e\in \E^{nd}} \text{ coker}( res_e) \cong \bigoplus\limits_{e\in \E^{nd}} \text{coker}( res_e).$$ 

	\end{proof}

\section{Geometric meaning of GKM-sheaf cohomology}\label{GMofC}

\begin{thm}
	\label{reflexive}
	Let X be a finite $T$-CW complex. Then $${H^{*}_{T}}(X)\cong H^{0}(\F_{X})$$ if and only if ${H^{*}_{T}}(X)$ is reflexive. 
\end{thm}
\begin{proof}
Suppose ${H^{*}_{T}}(X)\cong H^{0}(\F_{X}).$ By Corollary \ref{reflexive prop}, we conclude that $H_{T}^{*}(X) $ is reflexive. 

Conversely, suppose that ${H^{*}_{T}}(X)$ is reflexive. By Theorem \ref{sys thrm}, the Chang-Skjelbred sequence $$0\rightarrow H_T^*(X)\rightarrow H_T^*(X_0) \xrightarrow{\delta} H_T^{*+1}(X_1,X_0)$$ is exact, so ${H^{*}_{T}}(X)\cong ker(\delta)$ and $ker(\delta) \cong H^{0}(\F_{X})$ by Proposition \ref{fitting}. 
\end{proof}

	\begin{lem} 
		\label{phi}
	Let $X^{\prime}_{1}\subseteq X_1$ be the union of path components that do not intersect $X_0$. Suppose $H^*_T(X)$ is torsion free. Then $X^{\prime}_{1}=\emptyset.$
\end{lem}
\begin{proof}
	Observe that $X_{1}$ can be written as follows: 
	\begin{equation*}
	X_{1}=\bigcup_{H\leqslant T}X^{H}
	\end{equation*}
where the union is indexed by codimension one subtori $H$. Since $H_{T}^{\ast}(X)$ is torsion free, $H_{T}^{\ast}(X^{H})$ is free by Lemma \ref{lemma local}. By the Localization theorem \ref{locthm}, every path component of $X^{H}$ must intersect $X_{0}$ so $X^{H} \cap X^{\prime}_{1} = \emptyset$. We conclude $X^{\prime}_{1}=\emptyset.$
		\end{proof}
	\begin{lem}
		\label{main lemma}
		Suppose $H_T^{*}(X)$ is torsion free.\:Then $H^{0}(\V,\F_{X})\cong H_{T}^{\ast}(X_{0})$ and\\ 
		$H^{1}_{\E}(Top(\Gamma), \F_{X})\cong H_T^{*+1}(X_1,X_0)$.
	\end{lem}
\begin{proof}
The vertices $v \in \V$ corresponds  path components of $X_{0}$, so 
\begin{equation*}
H^{0}(\V,\F)=\bigoplus\limits_{v\in \V}\F_{v}=\bigoplus\limits_{v\in \V}H_{T}^{\ast}(v)=H_{T}^{\ast}(X_{0}).
\end{equation*}
The hyperedges $e \in \E$ for which $a(e) = \alpha$ correspond to path components of $X^{\ker(\alpha)}$ that intersect non-trivially with $X_{0}$. Combine Lemma \ref{decompose lemma} with Lemma \ref{phi} to get
\begin{equation}\label{1}
H_T^{*}(X_1,X_0)\cong \bigoplus_{e\in \E}H_T^{*}(e,e^{T}).
\end{equation}

\begin{claim}If $H_T^{*}(X)$ is torsion free, then $\F_{X}(U_{e})\cong H_T^{*}(e).$
\end{claim}

\begin{proof}
Recall that by definition $\F_{X}(U_{e}):=H_T^*(e)/t$ where $t$ is the torsion submodule so it is enough to show $H_T^*(e)$ is torsion free for all $e \in \mathcal{E}$ .  Since $H_T^{*}(X)$ is a submodule of finitely generated free $R$-module, it is torsion free. Apply Lemma \ref{lemma local}.
\end{proof}

	The restriction morphism $res_{e}: \F(U_{e})\rightarrow \F(I(e))$ is identical with the natural map $H_T^*(e)\rightarrow H_T^*(e^{T})$ which  is injective by the Localization Theorem \ref{locthm}. The long exact sequence for the pair $(e, e^{T})$ implies
	\begin{equation*} 
	H_T^{*+1}(e,e^{T})= \text{coker}(res_e). 
	\end{equation*}
Applying Proposition \ref{local e}, we have
	\begin{equation*}
	H^{1}_{\E}(Top(\Gamma),\F) = \bigoplus_{e\in \E^{nd}} H_T^{*+1}(e,e^{T}).
	\end{equation*} 
	Combining with (\ref{1}), we conclude
	\begin{equation*}
	H^{1}_{\E}(Top(\Gamma),\F) \cong H_T^{*+1}(X_1,X_0).
	\end{equation*}

\end{proof}

\begin{proof}[Proof of Theorem \ref{bigthm}]
	Since $H_T^{*}(X)$ is reflexive, Theorem \ref{reflexive} implies that
	\begin{equation*}
	0\rightarrow H^{0}(\F_{X})\rightarrow H_T^*(X_0) \xrightarrow{\delta} H_T^{*+1}(X_1,X_0)
	\end{equation*} 
	is exact. From Lemma \ref{main lemma} and (\ref{local equation}) we have an isomorphism of exact sequences.
\[
\xymatrix@C+1em@R+1em{ 
	0 \ar[r] & H_T^*(X) \ar[r] \ar^{\cong}[d] & H_T^*(X_0) \ar^{\delta}[r] \ar^{\cong}[d] & H_T^{*+1}(X_1,X_0) \ar[r] \ar^{\cong}[d] &  coker(\delta) \ar[r] \ar^{\cong}[d] & 0\\
	0 \ar[r] & H^{0}(\F_{X}) \ar[r] & H^{0}(\V,\F_X)  \ar[r] & H^{1}_{\E}(Top(\Gamma),\F_X) \ar[r] & H^{1}(\F_{X}) \ar[r] &  0 
}
\]

\end{proof}

\end{document}